\documentclass[10pt]{elsarticle}

\usepackage{lineno,hyperref}
\modulolinenumbers[5]

\journal{Operations Research Letters}
\usepackage{amsfonts,amssymb,amsmath,amsthm}
\everymath{\displaystyle}

\newcommand{\al}{\alpha}

\newcommand{\la}{\lambda}
\newcommand{\de}{\delta}

\newcommand{\bx}{\bar x}

\newcommand {\R} {\mathbb R}

\newcommand {\B} {\mathbb B}

\newcommand {\bd} {{\rm bd}\,}
\newcommand {\cone} {{\rm cone}\,}

\newcommand {\Int} {{\rm int}\,}

\newcommand{\norm}[1]{\left\Vert#1\right\Vert}
\newcommand{\abs}[1]{\left\vert#1\right\vert}

\newcommand{\ang}[1]{\left\langle #1 \right\rangle}

\newcommand{\AND}{\quad\mbox{and}\quad}
\newcounter{mycount}


\newtheorem{theorem}{Theorem}

\newtheorem{proposition}[theorem]{Proposition}
\newtheorem{lemma}[theorem]{Lemma}
\newtheorem{corollary}[theorem]{Corollary}
\newtheorem{corollary.pr}{Corollary}

\theoremstyle{definition}
\newtheorem{definition}[theorem]{Definition}
\theoremstyle{remark}
\newtheorem{remark}[theorem]{Remark}
\newtheorem{example}[theorem]{Example}

\usepackage{fullpage}
\usepackage{color}
\usepackage{minitoc}

\usepackage{bbm}
\usepackage{exscale,subfigure}
\usepackage{color}
\usepackage{hyperref}
\usepackage{graphicx}
\hypersetup{
	colorlinks=true,
	linkcolor=black, 
	citecolor=black, 
	urlcolor=black } 
\usepackage[weather,alpine,misc,geometry]{ifsym}

\newenvironment{subcondition}[1][\unskip]{%
	\par
	\noindent
	\textbf{Condition #1.}
	\noindent}
{}

\newcommand{\x}[1]{}
\usepackage{tikz}
\usetikzlibrary{patterns}
\usetikzlibrary{decorations.pathmorphing} 
\usetikzlibrary{matrix} 
\usetikzlibrary{arrows} 
\usetikzlibrary{calc} 
\usetikzlibrary{positioning,fit,calc} 
\usepackage{float}
\usetikzlibrary{arrows,decorations.markings,patterns}

\date{\today}
\usepackage[a4paper,bindingoffset=0.2in,%
left=0.5in,right=0.5in,top=0.5in,bottom=0.5in,%
footskip=.25in]{geometry}
\usepackage{multicol}
\begin{document}
	\begin{frontmatter}
		
		\title{A Note on the Finite Convergence of Alternating Projections}
		
		\author{Hoa T. Bui\fnref{myfootnote}}
		
		\fntext[myfootnote]{Corresponding author}
		\ead{hoa.bui@curtin.edu.au}
		\author{Ryan Loxton}	
		\ead{r.loxton@curtin.edu.au}
		\author{Asghar Moeini}
		\address{ARC Training Centre for Transforming Maintenance through Data Science, Curtin University, Australia
			}
		\address{School of Electrical Engineering, Computing and Mathematical Sciences,
			Curtin University, Australia}
		
		\begin{abstract}
			We establish sufficient conditions for finite convergence of the alternating projections method for two non-intersecting and potentially nonconvex sets. 
			Our results are based on a generalization of the concept of intrinsic transversality, which until now has been restricted to sets with nonempty intersection. 
			In the special case of a polyhedron and closed half space, our sufficient conditions define the minimum distance between the two sets that is required for alternating projections to converge in a single iteration.
		\end{abstract}
		
		
	\end{frontmatter}
\begin{multicols}{2}

\section{Introduction}\label{intro}

\noindent Throughout this paper, $X$ is a Hilbert space with inner product $\ang{\cdot,\cdot}$ and associated norm $\norm{\cdot}$. 
The method of {\it alternating projections} for two nonempty sets $A,B\subset X$ involves iterating the following steps, starting with $x_0\in A$:
\begin{eqnarray*}
	&&y_n \in P_B(x_n),\\
	&&x_{n+1} \in P_A(y_n).
\end{eqnarray*}
Here, $P_B(x_n)$ is the set of all projections of $x_n$ onto $B$ and $P_A(y_n)$ is the set of all projections of $y_n$ onto $A$. 
The study of the convergence of this method in the \emph{consistent case} (i.e. $A\cap B\neq \emptyset$) has a long history that can be traced back to von Neumman; see \cite{KruTha16,DruIofLew15,LewLukMal09,NolRon16} for historical comments. 
In particular, for convex settings,  Bregman \cite{Bre65} showed that the method always converges,
and a \emph{linear convergence} rate was established by Gubin et al. \cite{GubPolRai67} and Bauschke
\& Borwein \cite{BauBor93}.
For nonconvex settings, conditions such as \emph{superregularity} and \emph{intrinsic transversality} can be imposed to ensure linear convergence (see the results by Dao et al. \cite{dao2018linear} and Drusvyatskiy et al. \cite{DruIofLew15}). 
Furthermore, Noll and Rondepierre  in \cite{NolRon16} studied a general setting that allows
for nonlinear convergence under more subtle nonlinear regularity assumptions. 

For the inconsistent case, when $A\cap B =\emptyset$, the method of alternating projections does not converge to a single point, but under certain conditions it will converge to a pair of points of minimum distance. 
For example,
Cheney and Goldstein showed in \cite{CheGol59} that in Euclidean spaces when the two sets are closed and convex, and one of the sets is compact, the method converges and attains the minimum distance between the two sets. 
In particular, this result holds for two polytopes. 
Because of this important convergence property, alternating projections for inconsistent cases has been widely applied \cite{GubPolRai67,Com94,ComBon99}; see also \cite{YairAnd15, censor2018algorithms} for reviews. 

There has been a handful of research papers aimed at establishing certain convergence rates for the alternating projections method in the inconsistent case (see for example \cite{drusvyatskiy2017note}), but almost all results consider convergence in the limit rather than finite convergence. To the best of our knowledge the only exception is a recent paper by Behling et al. \cite{behling2020infeasibility}, which considers \emph{finite convergence} for two non-intersecting closed convex sets satisfying some error bound conditions. Our work in this paper also aims at finding sufficient conditions to ensure finite convergence, but for the general nonconvex setting.

Our approach is to extend the concept of intrinsic transversality, first defined for consistent cases in \cite{DruIofLew15},
to the more general setting when the intersection can be empty or nonempty (Conditions 1 and 2 in Section~\ref{main result}). 
Under either condition, we show that the number of iterations depends on the distance between the two sets, the starting point and the maximum angle between vectors of type $(a-b)$, where $a\in A$ and $b\in B$, and the proximal normal cones $N_A^{\text{prox}}(a)$ and $N_B^{\text{prox}}(b)$.
In particular, these results are applicable when the two sets are a polyhedron and a closed half space.
\begin{figure*}
	\centering
	\subfigure[Linear convergence]{\makebox[4cm][c]{\includegraphics[scale=0.6,width=1in]{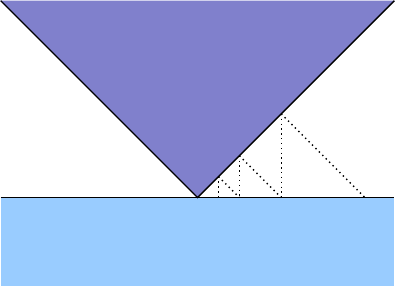}}}\quad\quad\quad\quad
	\subfigure[Finite convergence]{\makebox[4cm][c]{\includegraphics[scale=0.6,width=1in]{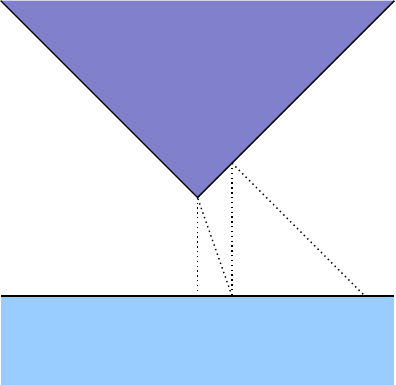}}}\quad\quad\quad\quad
	\subfigure[$1$ step convergence]{\makebox[4cm][c]{\includegraphics[scale=0.6,width=1in]{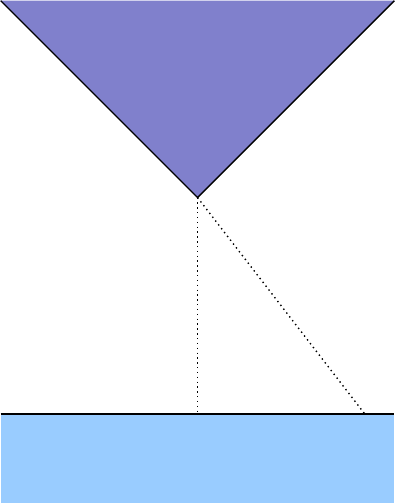}}}
	\caption{Convergence of the alternating projections method.}
	\label{Cn}
\end{figure*}

The alternating projections method can be used to solve linear programming problems. Indeed, a minimization linear program
can be formulated as finding the closest points of the following sets:
\begin{enumerate}
	\item[(1)] the problem's feasible region;  and
	\item[(2)] the closed half space containing all vectors  whose objective function value is strictly less than a specified lower bound for the objective over the feasible region.
\end{enumerate}

Finite convergence of the alternating projections method for this setting is guaranteed by our sufficient conditions.  This idea of solving linear programming problems using alternating projections was also independently considered in \cite{behling2020infeasibility}. However, our approach has some advantages. First, we can estimate the number of iterations required for convergence. Second, we can also determine the minimum distance (or lower bound), relative to the starting point that is needed to ensure convergence after one iteration (see Figure~\ref{Cn}).

The paper is organized as follows. In Section~\ref{Aux}, we recall some essential results that will be used in subsequent sections, and we provide a new proof for one of these results. Section~\ref{main result} contains our main results on the finite convergence of the alternating projections method. Finally, Section~\ref{Application}  explores the new finite convergence results for the special case of a polyhedron and a closed half space, which is related to linear programming.

\section{Preliminaries and auxiliary results} \label{Aux}
Let $\B$, $\B^*$ denote the open unit balls in the primal and dual spaces $X$, $X^*$ and furthermore let $B_\delta(x)$ and $\overline{B}_\delta(x)$ denote, respectively, the open and closed balls with center $x$ and radius $\delta>0$. We use $\R$ and $\R_+$ to denote the real line (with the usual norm) and the non-negative subset of the real line. We use the notation $\R_+(v)$ for the cone $\{\la v: \la\ge 0\}$ generated by a vector $v$ in $X$. 
The boundary and interior of a set $A$ are denoted as $\bd A$ and $\Int A$, respectively.
The distance from a point $x$ to a set $A$ is defined by $d(x,A):=\inf_{u \in A}\|u-x\|$, and we use the convention $d(x,\emptyset) = +\infty$.  The set of all projections of $x$ onto $A$ is 
$$
P_A(x):=\{a\in A: d(x,a)=d(x,A)\}.
$$
If $A$ is a closed subset of a finite dimensional space, then $P_A(x)\neq \emptyset$. 
Additionally, if $A$ is a closed convex set of an Euclidean space, then $P_A(x)$ is a singleton.
The identity $d\left(v/\norm{v},\R_+(u) \right) = d\left(u/\norm{u},\R_+(v) \right) $ for $u,v\in X\setminus\{0\}$ is used in subsequent sections. This assertion is proved as follows:
\begin{align*}
	\left[d\left(\frac{v}{\norm{v}},\R_+(u) \right)\right]^2 &= \min_{\la \ge 0} \norm{\frac{v}{\norm{v}}- \la u}^2\\
		& =  \min_{\la \ge 0} \left(1+\la^2\norm{u}^2-\frac{2\la}{\norm{v}}\ang{v,u} \right)\\
		& =\begin{cases}
		1- \frac{\ang{v,u}^2}{\norm{u}^2\norm{v}^2},& \text{if } \ang{u,v} >0,\\
		1,& \text{if } \ang{u,v} \le 0,
		\end{cases} 
\end{align*}
and similarly $$	\left[d\left(\frac{u}{\norm{u}},\R_+(v) \right)\right]^2 = \begin{cases}
	1- \frac{\ang{v,u}^2}{\norm{u}^2\norm{v}^2},& \text{if } \ang{u,v} >0,\\
	1,& \text{if } \ang{u,v} \le 0.
\end{cases} $$
As defined in \cite{Mor06.1}, the {\it proximal normal cone} to $A$ at $a\in A$ is:
\begin{align*}
N^{\text{prox}}_A(a)&:=\cone(P^{-1}_A(a)-a)\\
&=\left\{\lambda(x-a): \la \ge 0,a\in P_A(x) \right\}.
\end{align*}
For convenience, we will use the notation $N_A(a)$ instead of $N^{\text{prox}}_A(a)$ throughout. 
Observe that if $a\in P_A(x)$, then $x-a \in N_A(a)$. The proximal normal cone is related to the proximal subdifferential of a proper semicontinuous function $f$, denoted $\partial_P f$, or $\partial f$ for simplicity. Indeed, $\partial \mathbbm{1}_A(a) = N_A(a)$ for any $a\in A$, where here $\mathbbm{1}_A$ is the indicator function of the set $A$ defined by $\mathbbm{1}_A(x) = 0$ if $x\in A$ and $\mathbbm{1}_A(x) = +\infty$, otherwise. The proximal subdifferential satisfies the following fuzzy sum rule  (see \cite{Mor06.1}, page 240).
\begin{lemma}[Fuzzy sum rule]\label{fz}
Suppose $f_1$ is lower semicontinuous and
$f_2$ is Lipschitz continuous 
in a neighbourhood of $\bar x$.
Then, for any $x^*\in\partial(f_1+f_2) (\bar x)$ and $\varepsilon>0$, there exist $x_1,x_2\in X$ with $\|x_i-\bar x\|<\varepsilon$, $|f_i(x_i)-f_i(\bar x)|<\varepsilon$ $(i=1,2)$, such that
$
x^*\in \partial f_1(x_1) +\partial f_2(x_2)+ (\varepsilon\B^*).
$
\end{lemma}


The following result has been proved in \cite{DruIofLew15} and here we give an alternative proof. Our proof consists of two key ingredients: (1) Ekeland's Variational Principle \cite{Eke79}, and (2) the fuzzy sum rule in Lemma~\ref{fz}. 

\begin{theorem}[Distance decrease]\cite[Theorem 5.2]{DruIofLew15}
\label{sl}	
Consider a Hilbert space $X$, a closed set $A$, and points $a\in A$, $b\notin A$ with $\rho:=\norm{a-b}$ and $\al >0$. If there is $\delta>0$ such that 
	\begin{equation}\label{sl1}
	\inf\left\{d\left(\frac{b-x}{\norm{b-x}},N_A(x)\right): x\in B_\rho(b)\cap B_\delta(a) \cap A \right\}\ge\al,
	\end{equation}
	then $d(b,A)\le \norm{a-b}-\al\delta$.
\end{theorem}

\begin{proof}
	Consider the function $f(x)= \norm{x-b}$ and suppose to the contrary that \eqref{sl1} holds but
	$
	d(b,A)> \norm{a-b} -\al\de
	$.
	Take $\al'\in (0,\al)$ such that $d(b,A)> \norm{a-b} -\al'\de$.
	This is equivalent to $\inf_{x\in A} f(x) > f(a) - \al'\de$. 
	By Ekeland's Variational Principle, there is a vector $x_0\in A\cap B_\de(a)$ such that
	\begin{align}
	\label{1}
		f(x_0) &< f(a),\\
	\label{1.1}
		f(x_0)    &\le f(x) +\al'\norm{x-x_0},\quad \forall x \in A.
	\end{align}
	Due to \eqref{1}, $\norm{x_0-b}<\norm{a-b} =\rho$, or $x_0\in A\cap B_\rho(b)$. By \eqref{1.1}, it follows that $x_0$ is a global minimizer of the sum function $f(x)+\al'\norm{x-x_0}+\mathbbm{1}_A(x)$. Thus, 
	$$
	0 \in \partial\left(f(x)+\mathbbm{1}_A(x)+\al'\norm{x-x_0}\right)\mid_{x=x_0}.
	$$
	Take $\epsilon>0$ such that $$\epsilon<\min\left\{\al - \al', \rho-\norm{x_0-b}, \de - \norm{x_0-a}\right\}.$$
	By the fuzzy sum rule applied at $x_0$ for the functions $f(x)+\mathbbm{1}_A(x)$ and $\al'\norm{x-x_0}$, there exist $\bx,x' \in B_\epsilon(x_0)$ such that $\abs{f(\bx)+\mathbbm{1}_A(\bx) - f(x_0)-\mathbbm{1}_A(x_0)}<\epsilon$ and
	\begin{align*}
	0 &\in \partial (f(x) + \mathbbm{1}_A(x))\mid_{x=\bx}+ \partial(\al'\norm{x-x_0})\mid_{x=x'} +\epsilon \B^*\\
	 & \subset \partial (f(x) + \mathbbm{1}_A(x))\mid_{x=\bx}+ \al'\B^* +\epsilon \B^\ast \\
	 &\subset \partial (f(x) + \mathbbm{1}_A(x))\mid_{x=\bx} + \al \B^\ast.
	\end{align*}
	Therefore, $\bx \in A\cap B_\epsilon(x_0)$ and $\partial (f(x) + \mathbbm{1}_A(x))\mid_{x=\bx}\cap (\al \B^*) \neq \emptyset$.
	On the other hand, since $X$ is a Hilbert space and 
	$$f(\bx)= \norm{\bx-b} \ge d(b,A)>0,$$
	the function $f$ is differentiable at $\bx$ and  $\nabla f(\bx) = \frac{\bx-b}{\norm{b-\bx}}$. 
	Thus,
	$$\partial (f(x) + \mathbbm{1}_A(x))\mid_{x=\bx} = \frac{\bx-b}{\norm{b-\bx}}+ N_A(\bx).$$ 
	Recall that $\left(\frac{\bx-b}{\norm{b-\bx}}+ N_A(\bx)\right)\cap (\al\B^\ast) \neq \emptyset$, or
	$$
	d\left(\frac{b-\bx}{\norm{b-\bx}},N_A(\bx)\right)< \al.
	$$
	Furthermore, by the choice of $\epsilon$, $\bx \in A\cap B_\rho(b)$ and $$\norm{\bx - a }\le \norm{\bx - x_0}+\norm{x_0 - a}<\epsilon + \norm{x_0-a}<\delta,$$ 
	which shows that $\bx \in B_\de(a)$, and hence the previous inequality contradicts \eqref{sl1}.
\end{proof}

The next proposition is a supplementary result that provides a characterization for two points in disjoint convex sets that are of minimum distance apart.

\begin{proposition}
\label{c-con}
Consider two closed convex subsets $A,B$ of a Hilbert space $X$ with $d(A,B)>0$ and $a \in A, b\in B$. Then $\norm{a-b} =d(A,B)$ if and only if
	\begin{equation}\label{qual-1}
		(b-a) \in N_A(a),\AND (a-b) \in N_B(b).
	\end{equation}
\end{proposition}

\begin{proof}
		If $\norm{a-b} = d(A,B)>0$ with $a\in A$, $b\in B$, then $a \in P_A(b)$ and $b \in P_B(a)$. By the definition of the proximal normal cones, $b-a \in N_A(a)$ and $a-b \in N_B(b)$.
		
		Consider the distance function restricted to the sets $A$ and $B$ defined by $f(x,y) := \norm{x-y}+ \mathbbm{1}_{A\times B}(x,y)$, where $\mathbbm{1}_{A\times B}$ is the indicator function of the set $A\times B$. The product space $X\times X$ is equipped with the usual $L_2$ norm. 
		When the two sets $A,B$ are convex, the function $f$ is convex. The pair $(a,b)$ is the global minimizer of $f$, or equivalently a pair of shortest distance between $A$ and $B$, if and only if
		\begin{align*}
		0 \in \partial f(a,b) &= \partial \left(\norm{x-y}\right)\mid_{(x,y) = (a,b)} + N_{A\times B}(a,b) \\
		&= \partial \left(\norm{x-y}\right)\mid_{(x,y) = (a,b)} +N_A(a)\times N_B(b).
		\end{align*}
		Since $A\cap B =\emptyset$, then $a-b \neq 0$ and
		\begin{equation*}
		\partial (\norm{x-y})\mid_{(x,y)=(a,b)} =\left\{\left(\frac{a-b}{\norm{a-b}} , \frac{b-a}{\norm{b-a}} \right)\right\}.
		\end{equation*}
	The inclusion $0\in \partial f(a,b)$ is equivalent to \eqref{qual-1}.
\end{proof}

Note that in nonconvex settings inclusion \eqref{qual-1} is a necessary but not sufficient condition for $\norm{a-b} = d(A,B)$. 

The following definition was introduced in \cite{DruIofLew15}. Our convergence results rely on a modification of this definition to the inconsistent setting.

\begin{definition}[Intrinsic transversality]\cite[Definition 3.1]{DruIofLew15}
	\label{DEF}
	Given two closed sets $A,B$ of a Hilbert space $X$, $\bx\in A\cap B$, we say that $\{A,B\}$ is intrinsically transversal at $\bx$ with degree $\al\in (0,1)$ if there is $\rho>0$ such that for all $x\in (A\setminus B)\cap B_\rho(\bx), y\in (B\setminus A)\cap B_\rho(\bx)$, we have
	\begin{equation}\label{def2}
		\max\left\{d\left(\frac{x-y}{\norm{x-y}},N_B(y)\right),d\left(\frac{y-x}{\norm{x-y}},N_A(x)\right)\right\}\ge \al.
	\end{equation}
\end{definition}


The key result linking intrinsic transversality with linear convergence of the alternating projections method is restated from \cite{DruIofLew15} below.

\begin{theorem}[Linear convergence]\cite[Theorem 6.1]{DruIofLew15}
	\label{LC-Dru}
	If two closed sets $A,B$ of an Euclidean space $X$ are intrinsically transversal at a point $\bx \in A\cap B$, with degree $\al >0$, then, for any constant $c$ in the interval $(0,\al)$ the method of alternating projections, initiated sufficiently near $\bx$, converges to a point in the intersection $A\cap B$ with linear rate $1-c^2$.
\end{theorem}

\section{Convergence results}
\label{main result}
We extend the definition of intrinsic transversality in Definition~\ref{DEF} to more general frameworks without the assumption $A\cap B \neq\emptyset$, removing the need for $\bx$ and its local neighbourhood $B_\rho(\bx)$. 
Condition 1 below is a global condition that requires \eqref{def2} to hold across the  entire sets $A,B$.
Condition 1' is a weaker condition that only requires \eqref{def2} to hold in certain neighbourhoods around two points in $A,B$ that are of minimum distance apart.  

\begin{subcondition}[1]
	\label{def1}
	Given two closed sets $A,B$ of a Hilbert space $X$, and $\al \in (0,1)$, inequality \eqref{def2} holds for all $x\in A\setminus B, y\in B$ with $d(y,A)>d(A,B)$.
\end{subcondition}

\begin{subcondition}[1']
Given two closed sets $A,B$ of a Hilbert space $X$, $a\in A$, $b\in B$ such that $\norm{a-b}=d(A,B) =d$, and $\al \in (0,1)$, there exists $\rho>0$ such that inequality \eqref{def2} holds for all $x\in (A\setminus B)\cap B_{2d+\rho}(b)$ and  $y\in (B\setminus A)\cap B_{2d+\rho}(a)$ with $d(y,A)>d(A,B)$.
\end{subcondition}
\begin{remark}
	
If $A\cap B\neq \emptyset$, then $d(A,B)=0$ and Condition~1' reduces to Definition~\ref{DEF} and Condition~1 reduces to the following condition.

\begin{subcondition}[1'']
	Given two closed sets $A,B$ of a Hilbert space $X$, $A\cap B\ne \emptyset$, and a constant $\al \in (0,1)$, inequality \eqref{def2} holds for all $x\in A\setminus B, y\in B\setminus A$.
\end{subcondition}

\noindent Condition 1'' is an extension of Definition~\ref{DEF} to the global framework. Indeed, under Condition 1'', the pair $\{A,B\}$ is intrinsically transversal at any $\bx\in A\cap B$.
\end{remark}

We will show later in this section that under Conditions 1 or 1', when $A\cap B=\emptyset$, the method of alternating projections converges after a finite number of steps. 
To do this, we need the following key result.

\begin{lemma}\label{non-int2}
	Consider two closed subsets $A,B$ of a Hilbert space $X$, $x\in A\setminus B$ and $y \in P_B(x)$ satisfying $d(y,A)>d(A,B)$, and $\al \in (0,1)$, $\de := \al\norm{x-y}$. Suppose that the following inequality holds for all $z\in A\cap B_\delta(x)\setminus B$:
	\begin{equation}\label{def2z}
		\max\left\{d\left(\frac{z-y}{\norm{z-y}},N_B(y)\right),d\left(\frac{y-z}{\norm{z-y}},N_A(z)\right)\right\}\ge \al.
	\end{equation}
	 Then,
	\begin{equation}\label{k2}
	d(y,A)\le (1-\al^2)\norm{x-y}. 
	\end{equation}
\end{lemma}
\begin{proof}
Take $x\in A\setminus B$ and $y\in P_B(x)$ such that $d(y,A)>d(A,B)$.
We have $x\neq y$.
Let $\rho:=\norm{x-y}>0$ and $\delta:=\al\norm{x-y} >0$.
For any $z\in A\cap B_{\delta}(x)$, we have $z\neq y$ and
\begin{align*}
d\left(\frac{z-y}{\norm{z-y}},N_{B}(y)\right)&\le d\left(\frac{z-y}{\norm{z-y}},\R_+(x-y)\right)\\
&=  d\left(\frac{x-y}{\norm{x-y}},\R_+(z-y)\right)\\
&\le\norm{\frac{x-y}{\norm{x-y}} - \frac{z-y}{\norm{x-y}}}\\
&=\frac{\norm{z-x}}{\norm{x-y}}<\al.
\end{align*}
Therefore, by inequality \eqref{def2z},
\begin{equation}\label{A}
\inf\left\{d\left(\frac{y-z}{\norm{y-z}},N_A(z)\right): z\in A\cap B_{\delta}(x)\setminus B \right\}\ge\al.
\end{equation}  
Furthermore, for any $z\in A\cap B_\de(x)$, by the triangle inequality,
\begin{align*}
	d(z,B)&\ge d(x,B) - \norm{x-z}\\
	&= \norm{x-y}-\norm{x-z}\\
	&\ge \norm{x-y} - \de=(1-\al)\norm{x-y}  >0,
\end{align*}
which implies $z\in A\setminus B$.  Thus, $A\cap B_\de(x)\cap B_\rho(y)\setminus B = A\cap B_\de(x)\cap B_\rho(y)$, and from \eqref{A}, 
$$
\inf\left\{d\left(\frac{y-z}{\norm{y-z}},N_A(z)\right): z\in A\cap B_{\delta}(x)\cap B_\rho(y) \right\}\ge\al.
$$
Then, applying Theorem~\ref{sl}, we obtain
$$
d(y,A)\le \norm{x-y}-\al\delta=(1-\al^2)\norm{x-y},
$$
as required.\end{proof}

Note that Conditions~1  and 1'' meet the conditions required for Lemma~\ref{non-int2}. Now, consider three consecutive alternating projections:
\begin{equation*}
x_{2n}\in A,\quad  x_{2n+1}\in P_B(x_{2n}),\quad x_{2n+2}\in P_A(x_{2n+1}).
\end{equation*}
\noindent Under Condition 1, if $d(x_{2n+1},A) >d(A,B)$ (i.e., convergence has not been achieved after $2n+1$ iterations), then by Lemma~\ref{non-int2} we have
$$
\norm{x_{2n+2} - x_{2n+1}}= d(x_{2n+1},A) \le (1-\al^2)\norm{x_{2n+1} - x_{2n}}.
$$

\noindent This idea plays the core role in the following theorem.

\begin{theorem}\label{lin-con1}
	Consider two closed sets $A,B$ of a Hilbert space $X$ and suppose Condition 1 holds for some $\al \in (0,1)$. Consider a sequence of alternating projections $(x_n)$ where $x_{2n}\in A$ and $x_{2n+1} \in B$ ($n\ge 0$).
	\begin{enumerate}
		\item If $d(A,B)>0$, then the sequence $(x_n)$ attains the minimum distance in at most $2N+1$ steps, where
		\begin{equation}
		\label{constant}
		N:=\left\lfloor \log_{1-\al^2} \left(\frac{d(A,B)}{d(x_0,B)} \right)\right\rfloor.
		\end{equation}
		\item If $d(A,B) = 0$, then the sequence $(x_n)$ converges linearly to a point in the intersection $A\cap B$ with rate $(1-\al^2)$, i.e.,
		\begin{equation}\label{LN}
		\norm{x_{n+1} - x_{n}} \le (1-\al^2)^n \norm{x_1-x_0}.
		\end{equation}
	\end{enumerate}
	
\end{theorem}

\begin{proof}
Let $d(A,B)>0$. 
If convergence has not occurred after $2n+1$ steps ($n\ge 0$), then by induction,
\begin{equation}\label{induction}
d(A,B)<d\left(x_{2n+1},A\right)\le (1-\al^2)^{n+1} \norm{x_1-x_0}.
\end{equation}
Indeed, for $n=0$, \eqref{induction} follows from applying Lemma~\ref{non-int2} with $x_0$ and $x_1$, and assuming that \eqref{induction} holds for $n\ge 0$, if convergence has not occurred in $2(n+1)+1$ steps, then again by Lemma~\ref{non-int2},
\begin{align*}
	d(A,B)< d(x_{2n+3},A)
	 &\le (1-\al^2) \norm{x_{2n+3}-x_{2n+2}}\\
	& = (1-\al^2) d(x_{2n+2},B)\\
	&\le (1-\al^2)d(x_{2n+1},A)\\
	&\le (1-\al^2)^{n+2} \norm{x_1-x_0},
\end{align*}
which completes the induction argument. 
Now, if $d(A,B) < d(x_{2n+1},A)$, then from \eqref{induction}, 
$$n< \log_{1-\al^2} \left(\frac{d(A,B)}{d(x_0,B)} \right)-1 <N,$$ where $N$ is defined in \eqref{constant}.
Therefore, convergence must have occurred when $n=N$, which gives $2N+1$ as the upper bound for the number of iterations. 

Let $d(A,B) =0$. Observe that the condition $d(y,A)> d(A,B)$ is equivalent to $y\in B\setminus A$. Therefore, the sets $A$ and $B$ can be used interchangeably in Lemma~\ref{non-int2}. If the alternating projections have converged at step $n$, then
$$
\norm{x_{n+1}-x_n} = d(A,B)=0,
$$
and inequality \eqref{LN} holds trivially. If, on the other hand, the alternating projections have not converged at step $n$, then by Lemma~\ref{non-int2}, applied to $x_n$ and $x_{n-1}$, 
	\begin{align*}
	\norm{x_{n+1}-x_{n}}& \le (1-\al^2)\norm{x_{n}-x_{n-1}},
	\end{align*}
from which an induction argument proves \eqref{LN}.			
\end{proof}

The following examples demonstrate the application of Condition 1 in Theorem~\ref{lin-con1}.
\begin{figure*}[!ht]
	\centering
	\subfigure[Linear convergence (Example~\ref{ex1})]{\makebox[5cm][c]{\includegraphics[scale=0.6,width=1in]{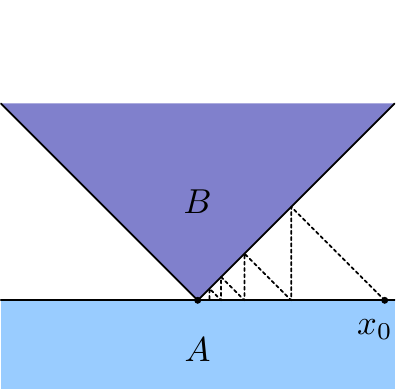}}}
	\quad\quad\quad\quad\quad\quad\quad\quad\quad\quad
	\subfigure[Finite convergence (Example~\ref{ex11})]{\makebox[5cm][c]{\includegraphics[scale=0.6,width=1in]{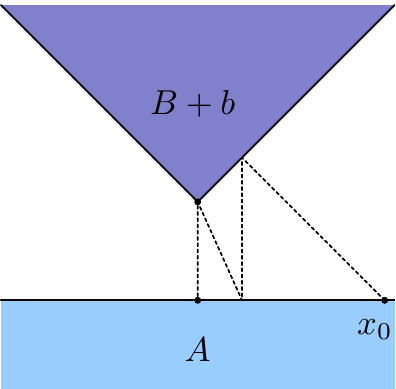}}}
	\caption{Alternating projections between a half space and the epigraph of $\abs{x}$.}
	\label{ln1}
\end{figure*}

\begin{example}\label{ex1}  

	Consider the space $X=\R^2$ equipped with the Euclidean norm and two closed sets $A:=\{(u,v): v\le 0\}$, $B:=\{(u,v): v\ge |u|\}$ and $x_0\in A$; {see Figure~\ref{ln1}(a)}. We show that Condition 1 holds in this setting.
		
	If $y \in \Int B$ or $x \in \Int A$, the proximal normal cones at these points are trivial, i.e., $N_B(y) =\{0\}$ or $N_A(x) =\{0\}$, and thus
		$d\left(\frac{x-y}{\norm{x-y}},N_B(y)\right)= 1$ or $d\left(\frac{y-x}{\norm{x-y}},N_A(x)\right)=1$, respectively. Hence, it is sufficient to consider $x\in \bd A$ and $y \in \bd B$. Take $x = (x_1,0)$  and $y= (x_2,\abs{x_2})$ with $x_1,x_2\in \R, x_2\neq 0$. Observe that
		$$
		N_A(x) =\R_+(0,1);
		\; N_B(y) = \begin{cases}
			\R_+(1,-1), & x_2 >0,\\
			\R_+(-1,-1), & x_2 <0.
		\end{cases}
		$$
		We have $$x-y= (x_1-x_2,-\abs{x_2}),\quad{\norm{x-y}} = {\sqrt{(x_1-x_2)^2+x_2^2}}.$$ 
		Therefore,
		\begin{align}	
			\notag
		d\left({y-x},N_A(x)\right)^2 &= \min_{t\ge 0} \left({\left(x_2-x_1\right)^2+\left( t -\abs{x_2}\right)^2} \right)\\
			& \label{P1E1}= \left(x_2-x_1\right)^2.
		\end{align}
		Moreover, for the case $x_2 >0$,
		\begin{align*}
		d\left({x-y},N_B(y)\right)^2 &=\min_{t\ge 0} \left(\left( t -(x_1-x_2)\right)^2+\left( t -x_2\right)^2 \right)\\
		&=\begin{cases}
			\tfrac{1}{2}\left(x_1-2x_2\right)^2, & x_1 \ge 0,\\
		(x_1-x_2)^2+x_2^2,& x_1 < 0;
		\end{cases}
		\end{align*}
		and for the case $x_2 <0$, 
			\begin{align*}
			d\left({x-y},N_B(y)\right)^2 &=\begin{cases}
				\tfrac{1}{2}\left(x_1-2x_2\right)^2, & x_1 \le 0,\\
				(x_1-x_2)^2+x_2^2,& x_1 > 0.
			\end{cases}
		\end{align*}
		Since $\left(x_1-2x_2\right)^2 = \left((x_1-x_2)-x_2\right)^2\le 2(x_1-x_2)^2+2x_2^2$, the following inequality always holds:
		$$
		d\left({x-y},N_B(y)\right)^2\ge \tfrac{1}{2} \left(x_1-2x_2\right)^2.
		$$
		This inequality combined with \eqref{P1E1} yields
		\begin{align*}
		&d\left(\frac{y-x}{\norm{y-x}},N_A(x)\right)^2 + d\left(\frac{x-y}{\norm{x-y}},N_B(y)\right)^2\\
		&\ge \frac{2\left(x_1-x_2\right)^2+\left(x_1-2x_2\right)^2}{{2(x_1-x_2)^2+2x_2^2}}\\
		&= \frac{(x_1-x_2)^2+[(x_1-x_2)^2+(x_1-2x_2)^2]}{{2(x_1-x_2)^2+2x_2^2}}\\
		&\ge  \frac{(x_1-x_2)^2+\tfrac{1}{2} x_2^2}{{2(x_1-x_2)^2+2x_2^2}}\ge  \frac{\tfrac{1}{2}(x_1-x_2)^2+\tfrac{1}{2} x_2^2}{{2(x_1-x_2)^2+2x_2^2}} = \frac{1}{4}.
		\end{align*}
		Hence,
		\begin{align*}
		\max\left\{d\left(\frac{y-x}{\norm{x-y}},N_A(x)\right),d\left(\frac{x-y}{\norm{x-y}},N_B(y)\right)\right\} &\\
		\ge\frac{1}{2\sqrt{2}}&.
		\end{align*}
		According to Theorem~\ref{lin-con1}(ii), the method of alternating projections converges linearly to the origin with rate $7/8$.
		
\end{example}
\begin{example}\label{ex11}
		Consider the setting in Example~\ref{ex1}.
		By shifting $B$ with a vector $b=(0,k)$, $k>0$, we obtain non-intersecting sets $\{A,B+b\}$ with $d(A,B+b) = \norm{b} =k$; see Figure~\ref{ln1}(b). 
		The points $(0,0)$ and $(0,k)$ are the closest points between the two sets. 
		Take $x\in \bd A$ and $y \in \bd (B+b)$ with $x=(x_1,0)$, $y = (x_2,\abs{x_2}+k)$, $x_1,x_2\in \R$, $x_2\neq 0$.
		We introduce the points $x'$ and $y'$ defined by $x'= (x_1+k,0)$  and $y' = (x_2+k,\abs{x_2}+k)$ if $x_2>0$, and $x'= (x_1-k,0)$  and $y' = (x_2-k,\abs{x_2}+k)$ if $x_2 <0$. Then clearly, $x'\in \bd A$, $y'\in \bd B$, $y'\neq(0,0)$, and $x'-y' = x-y$. Furthermore, by the results in  Example~\ref{ex1}, $N_A(x') = N_{A}(x)$, $N_B(y') = N_{B+b}(y)$, and
		{\small
		\begin{align*}
			&\max\left\{d\left(\frac{x-y}{\norm{x-y}},N_{B+b}(y)\right),d\left(\frac{y-x}{\norm{x-y}},N_A(x)\right)\right\}\\ 
			&= \max\left\{d\left(\frac{x'-y'}{\norm{x'-y'}},N_{B}(y')\right),d\left(\frac{y'-x'}{\norm{x'-y'}},N_A(x')\right)\right\}\\
			&\ge \frac{1}{2\sqrt{2}}.
		\end{align*}}%
		According to Theorem~\ref{lin-con1}(i), the alternating projections converge after $2\left\lfloor\log_{7/8} \left(\frac{k}{d(x_0,B+b)} \right)\right\rfloor+1$ steps. Note that for a fixed starting point $(x_0,0)$,
		when $k$ is sufficiently large, $P_{B+b}((x_0,0)) = \{(0,k)\}$, $d((x_0,0), B+b) = \sqrt{x_0^2+k^2}$, and only one step is required.
\end{example}


We now give two examples where Condition 1 is not satisfied.

\begin{example}\label{ex2}  
	Let $A:=\{(x,y):y\le 0\}$ and $B:=\{(x,y): y\ge x^2\}$; see Figure~\ref{nln}(a). Since $A$ and $B$ are convex, intrinsic transversality and subtransversality are equivalent \cite{Kru18}. Consider two sequences of points $(a_n)\subset A$ and $(b_n)\subset B$ defined by $a_n = (1/n,0)$ and $b_n = (1/n,1/n^2)$. For all $n\ge 1$,
		$$
		N_A(a_n) =\R_+(0,1);
		\; N_B(b_n) = \R_+(2/n,-1).
		$$
		(See normal cone of a function's epigraph in \cite{Mor06.1}.)
		We have 
		\begin{align*}
		d\left(\frac{b_n-a_n}{\norm{a_n - b_n}}, N_A(a_n)\right)&=0\\ d\left(\frac{a_n-b_n}{\norm{a_n - b_n}}, N_B(b_n)\right)&=\min_{t\ge 0}\big\|(0,-1)-t(2/n,-1)\big\|\\
		&\hspace{-1cm}= \min_{t\ge 0}\sqrt{(2t/n)^2+ (t-1)^2}\le \frac{2}{n}.
		\end{align*}
		Therefore,
		\begin{align*}
			\max\left\{d\left(\frac{b_n-a_n}{\norm{a_n - b_n}}, N_A(a_n)\right),d\left(\frac{a_n-b_n}{\norm{a_n - b_n}}, N_B(b_n)\right)\right\} &\\
		\le\frac{2}{n}.
		\end{align*}
			Since the right-hand side approaches $0$ as $n\rightarrow\infty$, intrinsic transversality (and hence subtransversality) does not hold for any $\alpha\in(0,1)$, and therefore Condition~1 also does not hold. The method of alternating projections cannot converge linearly because subtransversality is a necessary condition for linear convergence \cite[Theorem 8]{ luke2020necessary}.
\end{example}
\begin{example}\label{ex22}
		Consider the setting in Example~\ref{ex2}.
		Shifting $B$ by $b=(0,k)$, $k >0$, we obtain $\{A,B+b\}$ with closest points $(0,0)\in A$ and $(0,k)\in (B+b)$; see Figure~\ref{nln}(b). 
		Suppose that the alternating projections have not converged to the minimum distance in $n\ge 0$ steps. If $x_n = (v_1,v_1^2+k) \in B+b$ with $v_1 \neq 0$, then $x_{n+1}= (v_1,0) \neq (0,0)$. If $x_n = (v_1,0) \in A$ with $v_1 \neq 0$, then $b\notin P_{B+b}(x_n)$ since $d\left(x_n - b, N_{B+b}(b) \right) = d\left((v_1, - k), \R_+(0,-1) \right)>0$, and hence $x_{n+1}\neq (0,k)$.	
		Therefore, if the alternating projections have not converged after $n$ steps, then they also do not converge after $n+1$ steps, and there is no finite convergence.
\end{example}

\begin{figure*}[!ht]
	\centering
	\subfigure[No linear convergence (Example~\ref{ex2})]{\makebox[5.5cm][c]{\includegraphics[width=1in]{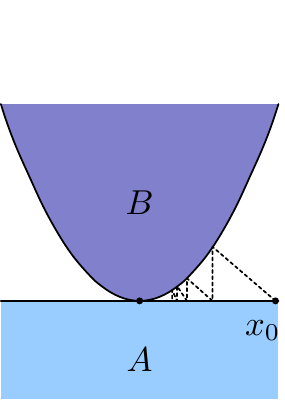}} }\quad\quad\quad\quad\quad\quad\quad\quad\quad\quad\quad\quad
	\subfigure[No finite convergence (Example~\ref{ex22})]{\makebox[5.5cm][c]{\includegraphics[width=1in]{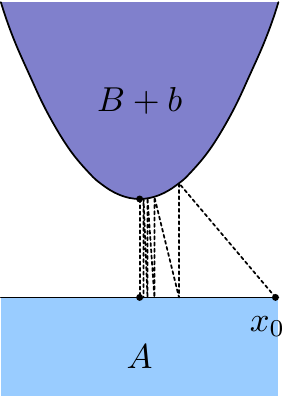}}}
	\label{nln}
	\caption{Alternating projections between a half space and the epigraph of $x^2$.}
\end{figure*}



%


The next theorem establishes linear and finite convergence for the local setting defined by Condition 1'.

\begin{theorem}\label{lin-con}
	Consider two closed subsets $A,B$ of a Hilbert space $X$, $a\in A, b\in B$ such that $\norm{a-b}= d(A,B)=d$ and suppose Condition 1' holds for some $\al \in (0,1)$ and $\rho >0$. Consider a sequence of alternating projections $(x_n)$ where $x_{2n}\in A$ and $x_{2n+1} \in B$ ($n\ge 0$), initiated sufficiently close to $a$. 
	\begin{enumerate}
		\item If $d(A,B)>0$, then $(x_n)$ attains the minimum distance in one step.
		\item If $d(A,B) =0$, then $(x_n)$ converges linearly with rate $1-\al^2$.
	\end{enumerate}
\end{theorem}

\begin{proof}
Suppose that Condition 1' holds with $\al \in (0,1)$ and $\rho>0$.

Consider the case $d(A,B) >0$. Take $\kappa:=\min\{\al^2d(A,B),\rho/2\}$ and $x_0 \in A\cap B_\kappa(a)$, $x_1 \in P_B(x_0)$. Then 
\begin{align*}
\norm{x_1 - a} & \le \norm{x_1 - x_0} + \norm{x_0-a}\\
&\le \norm{b-x_0}+\norm{x_0-a} \\
&\le \norm{b-a} + 2\norm{x_0-a} < 2d(A,B) +2\kappa,
\end{align*}
and so $x_1\in B_{2d+\rho}(a)$. Let $\de := \al\norm{x_0-x_1}$.
For any $z\in A\cap B_\de(x_0)$,
\begin{align*}
	\norm{z-b} &\le \norm{z-x_0}+\norm{x_0-b}\\
	&\le \al\norm{x_0-x_1}+\norm{x_0-b} \\
	&\le (1+\al)\norm{x_0-b}\\
	&< 2d(A,B) + 2\kappa,
\end{align*}
and hence $z\in A\cap B_{2d+\rho}(b)$. Thus, since $x_1\in B_{2d+\rho}(a)$ and $z\in A\cap B_\de(x_0)$ implies $z\in A\cap B_{2d+\rho}(b)$, Condition 1' ensures that Lemma~\ref{non-int2} can be applied to $x_0$ and $x_1$ when $d(x_1,A) > d(A,B)$.
This gives
\begin{align*}
	d(x_1,A)&\le  \norm{x_1-x_0}-\al^2 \norm{x_1-x_0}\\
&\le \norm{x_0-b} - \al^2 d(A,B)\\
&\le \norm{x_0-b}-\kappa\\
&\le \norm{x_0-a}+\norm{a-b}-\kappa<d(A,B),
\end{align*}
which contradicts the assumption $d(x_1,A) >d$. This implies $d(x_1,A) = d(A,B)$, and hence convergence occurs after one step when the sequence is initiated in $A\cap B_\kappa(a)$.
		
Now consider the case $d(A,B) = 0$. In this case, we have $a\equiv b \equiv \bx \in A\cap B$.
Set $\kappa :=\al^2\rho/(2\al^2+2)$.
Take $x_0\in A\cap B_{\kappa}(\bx)$. We now prove by induction that the following holds for all $n\ge 0$:
\begin{equation}\label{local}
x_n\in B_{\rho/2}(\bx),\quad \norm{x_{n+1} - x_n} \le (1-\al^2)^n\norm{x_1-x_0}.
\end{equation}
This holds immediately for $n=0$. Assume now that it holds for $n=0,\ldots,k$. Then, 
\begin{align*}
\norm{x_{k+1}-\bx} &\le \sum_{n=0}^{k}\norm{x_{n+1}-x_{n}}+\norm{x_0-\bx}\\
&\le \sum_{n=0}^{k}(1-\al^2)^n\norm{x_{1}-x_{0}}+\norm{x_0-\bx}\\
&= \left(\frac{1-(1-\al^2)^{k+1}}{\al^2}\right)\norm{x_1-x_0}+\norm{x_0-\bx}\\
&\le\frac{\norm{x_1-x_0}}{\al^2} +\norm{x_0-\bx}\\
&\le \left(\frac{1}{\al^2}+1 \right)\norm{x_0-\bx}< \tfrac{1}{2}\rho,
\end{align*}
and hence $x_{k+1}\in B_{\rho/2}(\bx)$. If $x_{k+1}\in A\cap B$, then convergence has been achieved and the inequality in \eqref{local} holds trivially for $n=k+1$. Otherwise, $d(x_{n+1},A)>0$ if $x_k\in A$ and $d(x_{k+1},B)>0$ if $x_k\in B$. For any $z\in B_\de(x_k)$, with $\de := \al\norm{x_{k+1} - x_k}$, we have
\begin{align*}
\norm{z-\bx} &\le \norm{z-x_k} + \norm{x_k-\bx} \\
&< \de + \norm{x_k-\bx}\\
& = \al \norm{x_{k+1} - x_k} + \norm{x_k - \bx}\\
&\leq 2\norm{x_k - \bx}< \rho,
\end{align*}
and hence $z\in B_\rho(\bx)$. Hence, by Condition~1', if $x_{k+1}\notin A\cap B$, then when $x_k\in A$, we can apply Lemma~\ref{non-int2} to derive $d(x_{k+1},A) \le (1-\al^2)\norm{x_{k+1}-x_k}$, and when $x_k\in B$ we can apply Lemma~\ref{non-int2} to derive $d(x_{k+1},B) \le (1-\al^2)\norm{x_{k+1}-x_k}$. This gives
\begin{align*}
\norm{x_{k+2} - x_{k+1}} &\le (1-\al^2)\norm{x_{k+1} - x_k}\\
&\le (1-\al^2)^{n+1}\norm{x_1-x_0},
\end{align*}
which shows that \eqref{local} holds for $n=k+1$.
\end{proof}
\begin{remark}
	\begin{enumerate}
		\item Since Condition 1' covers intrinsic transversality, Theorem~\ref{lin-con} covers Theorem~\ref{LC-Dru}.
		\item Under Condition 1' when $d(A,B)>0$, it follows from Theorem~\ref{lin-con}(i) that if the alternating projections converge to $a$ and $b$, then they must converge in a finite number of steps, because eventually the sequence will enter the ball $A\cap B_\kappa(a)$, after which only one more projection is needed. However, estimating the number of steps is difficult because Condition 1'  is only a local condition and it may not be satisfied at every iteration. In Theorem~\ref{lin-con1}, we could quantify the number of steps because Condition 1 applies globally, unlike Condition 1'.
	\end{enumerate}

\end{remark}

We now define an alternative to Conditions~1 and~1'.
\begin{subcondition}[2]
	\label{con3}
	Given two closed subsets $A,B$ of a Hilbert space $X$, and $\al \in (0,1)$, $\beta \in [0,1]$, the following inequality holds for all $x\in A\setminus B$ and $y\in P_B(x)$ such that $d(y,A)>d(A,B)$:
	\begin{align} \nonumber
	&\inf \bigg\{ d\left(\frac{y-z}{\norm{y-z}}, N_A(z) \right):\; z\in A\cap B_\rho (y) \cap B_\de(x),\\
	\nonumber
	&\quad\quad\quad\rho = \norm{x-y},\,\de= \al(\norm{x-y}-\beta d(A,B))\bigg\}\ge \al.
	\end{align}
\end{subcondition}

Unlike Condition 1, which considers all vectors $y \in B\setminus A$, Condition~2 only considers the projections of $x\in A$ onto $B$. 
Condition~2 also provides the flexibility to choose the neighbourhood of $x$ in $A$ by adjusting the new parameter $\beta$. The next theorem shows that the speed of convergence depends on the neighbourhood's radius. 

\begin{theorem}\label{lin-con2}
	Consider two closed subsets $A,B$ of a Hilbert space $X$ with $d(A,B)=d\ge 0$ and suppose Condition~2 holds for some $\al \in (0,1)$ and $\beta \in [0,1]$. Consider a sequence of alternating projections $(x_n)$ where $x_{2n}\in A$ and $x_{2n+1} \in B$ ($n\ge 0$). 
	\begin{enumerate}
		\item If $\beta <1$ and $d>0$, then $(x_n)$ attains the minimum distance in at most $2N+1$ steps, where
		\begin{equation}\label{N}
			N := \left\lfloor \log_{1-\al^2}\left(\frac{d  (1-\beta)}{ \norm{x_1-x_0} - \beta d}\right) \right \rfloor.
		\end{equation}
		\item If $d=0$ or $\beta=1$, then the sequence $(x_n)$ converges linearly with rate $\sqrt{1-\al^2}$.
	\end{enumerate}
\end{theorem}

\begin{proof}
	Assume that Condition 2 holds for $\al \in (0,1)$ and $\beta \in [0,1]$. 
	If $d(x_{2n+1},A)>d(A,B)$, then Condition 2 ensures that we can apply Theorem~\ref{sl} with $x_{2n}\in A$ and $x_{2n+1}\in B$ to yield
	\begin{align}
		\label{sub-P1}
		d<d(x_{2n+1},A)
		\le (1-\al^2)\norm{x_{2n+1} - x_{2n}}+\al^2\beta d.
	\end{align}
	We now prove by induction that, whenever $d(x_{2n+1},A)> d(A,B)$ ($n\ge 0$), 
	\begin{equation}
		\label{intd}
		d<d(x_{2n+1},A)\le (1-\al^2)^{n+1}\left( \norm{x_{1}-x_{0}} - \beta d \right)+\beta d.
	\end{equation}
	This is easily proved in the base step by substituting $n=0$ into \eqref{sub-P1}.
	For the inductive step, we assume that \eqref{intd} holds for $n=k\ge 0$, and then if $d(x_{2k+3},A) > d(A,B)$, using \eqref{sub-P1} gives
		{
		\begin{align*}
			\notag
			d(x_{2k+3},A)
			&\le (1-\al^2)\norm{x_{2k+3} - x_{2k+2}}+\al^2\beta d\\
			&\le (1-\al^2)\norm{x_{2k+2} - x_{2k+1}}+\al^2\beta d\\
			&= (1-\al^2)d(x_{2k+1},A)+\al^2\beta d\\
			&\le (1-\al^2)^{k+2}\left( \norm{x_{1}-x_{0}}-\beta d\right) + (1-\al^2) \beta d\\
			& \quad+\al^2\beta d\\
			&= (1-\al^2)^{k+2}\left( \norm{x_{1}-x_{0}} - \beta d \right) + \beta d,
	\end{align*}}%
	which proves \eqref{intd} for $n=k+1$, and hence \eqref{intd} holds for all $n\ge 0$. We now consider two cases:
	\begin{enumerate}
	\item[1.] $\beta <1$ and $d>0$; and
	\item[2.] $\beta =1$ or $d=0$.
	\end{enumerate}
	For case 1, if the alternating projections have not converged in $2n+1$ iterations ($d(x_{2n+1},A) >d$), then by \eqref{intd},
		\begin{align*}
		n < \log_{1-\al^2}\left(\frac{d (1-\beta)}{ \norm{x_1-x_0} - \beta d}\right)-1 < N.
		\end{align*}
	Hence, convergence must have occurred  after $2N+1$ iterations, proving part (i).

	For case 2, we have from \eqref{intd} that while convergence has not occurred, for odd integers $n\ge 1$,
	\begin{align}
		\notag
		\norm{x_{n+1}-x_n} &=d(x_n,A)\\
		\notag
		&\le\left(\sqrt{1-\al^2}\right)^{n+1}(\norm{x_1-x_0}-d)+d\\
		\label{str} 
		&<\left(\sqrt{1-\al^2}\right)^{n}(\norm{x_1-x_0}-d)+d,
	\end{align}
	and for even integers $n\ge 2$,
	\begin{align}
		\notag
		\norm{x_{n+1}-x_n} &=d(x_n,B)\\
		\notag
		&\le \norm{x_n - x_{n-1}}\\
		\notag
		&=d(x_{n-1},A)\\
		\label{estr} &<\left(\sqrt{1-\al^2}\right)^{n}(\norm{x_1-x_0}-d)+d.
	\end{align}
	Inequalities \eqref{str} and \eqref{estr} show that $\norm{x_{n+1}-x_n}$ converges to $d$ linearly with rate $\sqrt{1-\al^2}$.
\end{proof}

\begin{remark}
	In part (i) of Theorem~\ref{lin-con2}, the neighbourhood radius $\de$ in Condition 2 is always larger than the constant $\al(1-\beta) d(A,B)>0$, and in this case finite convergence is guaranteed. In part (ii), $\de \to 0$ as $\norm{x-y}\to d(A,B)$ and hence the neighbourhoods become arbitrarily small, and the theorem only gives convergence in the limit. We suspect that it may be possible to derive stronger convergence rates by using different expressions for the radius $\de$ in Condition 2.
	
	When $d(A,B) =0$, Theorem~\ref{lin-con2} provides a new sufficient condition for linear convergence in the consistent case. Condition~2 is weaker than intrinsic transversality as it only takes into account the normal cones for one of the sets. This shows that intrinsic transversality is not a necessary condition for linear convergence of alternating projections in general nonconvex settings.
\end{remark}

\section{Special case: Polyhedron and closed half space}
\label{Application}
We now apply the results in Section~\ref{main result} to the special case where the two sets are a polyhedron and a closed half space.

\begin{proposition}
\label{active}
	Consider two non-intersecting closed subsets $A,B\subseteq\mathbb{R}^{n}$ ($n\ge 1$) defined by
	$$A:=\{x: \ang{c,x}\le M \}, \AND B:=\{x: \mathbb{A} x\le b\},$$
	where $\mathbb{A}$ is a $m\times n$ matrix with rows $a_i\in \R^n$ ($i=1,\ldots,m$), $b\in \R^m$, $c\in \R^n\setminus\{0\}$ and $M\in \R$. 
	Then, the pair $\{A,B\}$ satisfies Condition 2 with $\beta =0$ and 
	\begin{equation}\label{al}
	\al:=\frac{1}{2} \min\left\{1,\min_{\substack{
		i=1,\ldots,m;\\ \ang{a_i,c}>-\norm{a_i}\norm{c} 
	}} d \left(\frac{a_i}{\norm{a_i}},\R_{+}(-c)\right) \right\},
	\end{equation}
	with the convention $\min\emptyset = +\infty$.
\end{proposition}

\begin{proof}
Take $x\in A$ and $y \in P_B(x)$ such that $d(y,A) > d(A,B)$ and let $x'\in P_A(y)$.
Then, $x-y \in N_B(y)$ and $y-x'\in N_A(x') = \R_+(c)$. Furthermore, since $A$ and $B$ are closed convex sets and $y-x'\in N_A(x')$, we must have $x'-y\notin N_B(y)$, since otherwise by Proposition~\ref{c-con} $d(y,A) = d(A,B)$, which is a contradiction.
Therefore, $-c\notin N_B(y)$, or equivalently, 
\begin{equation}\label{ax}
\ang{c,v} > - \norm{c}\cdot\norm{v},\quad \text{for all } v\in N_B(y)\setminus\{0\},
\end{equation}
since for vectors $u,v\in \mathbb{R}^{n}$ with $\norm{u}=\norm{v} = 1$, we have $\ang{u,v}=-1$ if and only if $u=-v$.
By \cite[Theorem 6.46]{rockafellar2009variational}, we have 
$$
N_B(y) = \cone\{a_i:\; i\in I(y) \},
$$
where $I(y) := \{i:\,\ang{a_i,y} = b_i\}$. Note that if $i \in I(y)$, then \eqref{ax} implies $\ang{a_i,c}> - \norm{a_i}\cdot\norm{c}$.
 Hence, $$x-y\in N_B(y) \subset \cone\{a_i:\; \ang{a_i,c}>-\norm{a_i}\cdot\norm{c}\}=: S,$$ 
 and from \eqref{al}, we obtain
 \begin{align*}
d\left(\frac{y-x}{\norm{x-y}}, \R_+(c) \right) &= d\left(\frac{x-y}{\norm{x-y}}, \R_+(-c) \right)\\
&\ge \inf\left\{d\left(\frac{v}{\norm{v}}, \R_+(-c) \right): v\in S\right\}\\
& = \min_{\substack{
		i=1,\ldots,m;\\ \ang{a_i,c}>-\norm{a_i}\norm{c} 
}}d \left(\frac{a_i}{\norm{a_i}},\R_{+}(-c)\right)\\
 &\ge 2\al.
 \end{align*}
Set $\de:= \al\norm{x-y}$, $\rho :=\norm{x-y}$. Take $z\in A\cap B_\de(x)\cap B_\rho(y)$.
Observe that either $z\in \Int A$ and $N_A(z) = \{0\}$ or $z\in \bd A$ and $N_A(z) = \R_{+}(c)$.
Thus,
\begin{align*}
d\left(\frac{y-z}{\norm{y-z}}, N_A(z) \right) &\ge \frac{\norm{y-x}}{\norm{y-z}} d\left(\frac{y-z}{\norm{y-x}}, \R_+(c) \right)\\
&\ge d\left(\frac{y-x+x-z}{\norm{y-x}}, \R_+(c) \right)\\
&\ge d\left(\frac{y-x}{\norm{y-x}}, \R_+(c) \right) -\frac{\norm{x-z}}{\norm{y-x}}\\
& \ge 2\al - \al = \al.
\end{align*}
Hence, the pair $\{A,B\}$ satisfies Condition 2 with $\beta=0$ and $\al$ defined by \eqref{al}.

Note that the above derivations assume $S \neq \emptyset$, since otherwise no such $x\in A$ and $y\in P_A(x)$ with $d(y,A) > d(A,B)$ exist, and Condition 2 is redundant.
\end{proof}

\begin{corollary}
\label{poly-sub}
Let $A$ and $B$ be as defined in Proposition~\ref{active} with $d(A,B) >0$ and let $\al$ be defined by \eqref{al}. Then, the method of alternating projections, initiated at $x_0 \in A$, attains the minimum distance in at most $2N+1$ steps, where
$$
N: =  \left\lfloor\log_{1-\al^2}\left(\frac{d(A,B)}{d(x_0,B)} \right) \right \rfloor.
$$
Furthermore, if $d(x_0,B)<\frac{d(A,B)}{1-\al^2}$, then the minimum distance is attained after one step.
\end{corollary}

\begin{proof}
	Observe that the constant $\al$ given in \eqref{al} is always in $(0,1)$, since $a_i \notin \R_+(-c)$ for $a_i\in S$, and hence
	$$
	0< d(a_i/\norm{a_i}, \R_+(-c)) \le 1,\quad \text{ for each } a_i \in S,
	$$
	where $S$ is as defined in the proof of Proposition~\ref{active}.
	
	By Proposition~\ref{active} and Theorem~\ref{lin-con2}, we conclude that the alternating projections converge after $2N+1$ steps with
	$$
	N = \left\lfloor \log_{1-\al^2}\left(\frac{d(A,B)}{d(x_0,B)} \right) \right\rfloor.
	$$
	When $d(x_0,B)< \frac{d(A,B)}{1-\al^2}$, we have $N=0$ and the method converges after one step.
\end{proof}

\begin{remark}\label{rm19}
\begin{enumerate}
	\item Consider two closed convex sets $A$, $B$ and let $a\in A$, $b\in B$ such that $\norm{a-b} = d(A,B)>0$. Then by the definition of normal cone, $b-a \in N_A(a)$ and $a-b \in N_B(b)$. If we shift $A$ by a vector $v := \la(b-a)$ with $\la>0$, then $b-(a-v) \in N_A(a) = N_{A-v}(a-v)$ and $a-v-b\in N_B(b)$. Hence, by Proposition~\ref{c-con}, we have $\norm{(a-v) - b} = d(A-v,B)$. This result is used in the next remark to determine by how much the closed half space needs to be shifted to ensure one-step convergence. 
	\item Consider two sets $A,B$ as defined in Proposition~\ref{active}, and $x_0\in A$. Let $a\in A$ and $b\in B$ with $\norm{a-b} = d(A,B)$, and then $c = \la(b-a)$ for some $\la >0$, since $b-a \in N_A(a) = \R_+(c)$. To ensure one-step convergence, we can shift $A$ by a vector $v := \mu c$, where $\mu>0$ and $\mu> \tfrac{1}{\al^2\norm{c}}\left((1-\al^2)d(x_0,B) - d(A,B) \right)$. Then by part (i) above, $d(A-v,B) = d(A,B)+\norm{v} = d(A,B)+\mu\norm{c}$. From the choice of $\mu$, we have
	$$
	\mu\norm{c}+d(x_0,B) < \frac{d(A,B) + \mu\norm{c}}{1-\al^2}.
	$$
	Therefore,
	$d(x_0-v,B) \le d(x_0,B)+\norm{v} = d(x_0,B) + \mu\norm{c } < \frac{d(A,B) + \mu\norm{c}}{1-\al^2} = \frac{d(A-v,B)}{1-\al^2}$. By Corollary~\ref{poly-sub}, the alternating projections for $A-v$ and $B$, starting from $x_0-v$, converge after one step.
\end{enumerate}
\end{remark}
We now propose a projection method for solving linear programming problems of the form (LP): $\min_{x:\,\mathbb{A}x\leq b}\ang{c,x}$,
where $c\in \R^n$ and $\mathbb{A}$ is a matrix. We assume that LP is bounded with $M$ as a lower bound that is strictly less than the optimal value. Set 
\begin{align*}
A&:=\{x\in \R^n: \ang{c,x}\le M \},\\
B&:=\{x\in \R^n: \mathbb{A}x \le b\}.
\end{align*}
A solution of LP is obtained by applying iteratively the alternating projections $P_B(x_{2n})$, $P_{A}(x_{2n+1})$ until the minimum distance is attained. By Remark~\ref{rm19}(ii), 
the projection of $x_0-\mu c$ with $\mu > \frac{(1-\al^2)}{\al^2\norm{c}}d(x_0,B)$  onto $B$ is a solution of LP.

\section{Acknowledgement}
The authors are supported by the Australian Research Council through the Centre for Transforming Maintenance through Data Science (grant number IC180100030).  We are also grateful to the referees for their constructive comments.
\end{multicols}
\bibliography{BUCH-kr,Kruger,KR-tmp,Hoa} 

\def\cprime{$'$} \def\cftil#1{\ifmmode\setbox7\hbox{$\accent"5E#1$}\else
  \setbox7\hbox{\accent"5E#1}\penalty 10000\relax\fi\raise 1\ht7
  \hbox{\lower1.15ex\hbox to 1\wd7{\hss\accent"7E\hss}}\penalty 10000
  \hskip-1\wd7\penalty 10000\box7} \def\cprime{$'$} \def\cprime{$'$}
  \def\cprime{$'$} \def\cprime{$'$} \def\cprime{$'$}
  \def\Dbar{\leavevmode\lower.6ex\hbox to 0pt{\hskip-.23ex \accent"16\hss}D}
  \def\cfac#1{\ifmmode\setbox7\hbox{$\accent"5E#1$}\else
  \setbox7\hbox{\accent"5E#1}\penalty 10000\relax\fi\raise 1\ht7
  \hbox{\lower1.15ex\hbox to 1\wd7{\hss\accent"13\hss}}\penalty 10000
  \hskip-1\wd7\penalty 10000\box7} \def\cprime{$'$}
\begin{thebibliography}{10}
\expandafter\ifx\csname url\endcsname\relax
  \def\url#1{\texttt{#1}}\fi
\expandafter\ifx\csname urlprefix\endcsname\relax\def\urlprefix{URL }\fi
\expandafter\ifx\csname href\endcsname\relax
  \def\href#1#2{#2} \def\path#1{#1}\fi

\bibitem{KruTha16}
A.~Y. Kruger, N.~H. Thao, Regularity of collections of sets and convergence of
  inexact alternating projections, J. Convex Anal. 23~(3) (2016) 823--847.

\bibitem{DruIofLew15}
D.~Drusvyatskiy, A.~D. Ioffe, A.~S. Lewis, Transversality and alternating
  projections for nonconvex sets, Found. Comput. Math. 15~(6) (2015)
  1637--1651.

\bibitem{LewLukMal09}
A.~S. Lewis, D.~R. Luke, J.~Malick, Local linear convergence for alternating
  and averaged nonconvex projections, Found. Comput. Math. 9~(4) (2009)
  485--513.

\bibitem{NolRon16}
D.~Noll, A.~Rondepierre, On local convergence of the method of alternating
  projections, Found. Comput. Math. 16~(2) (2016) 425--455.

\bibitem{Bre65}
L.~M. Bregman, {The method of successive projection for finding a common point
  of convex sets.}, {Sov. Math., Dokl.} 6 (1965) 688--692.

\bibitem{GubPolRai67}
E.~Kopeck\'a, S.~Reich, A note on the {v}on {N}eumann alternating projections
  algorithm, Journal of Nonlinear and Convex Analysis 5 (2004) 379--386.

\bibitem{BauBor93}
H.~H. Bauschke, J.~M. Borwein, On the convergence of von {N}eumann's
  alternating projection algorithm for two sets, Set-Valued Anal. 1~(2) (1993)
  185--212.

\bibitem{dao2018linear}
M.~N. Dao, H.~M. Phan, Linear convergence of the generalized
  {D}ouglas-{R}achford algorithm for feasibility problems, Journal of Global
  Optimization 72~(3) (2018) 443--474.

\bibitem{CheGol59}
W.~Cheney, A.~Goldstein, Proximity maps for convex sets, Proceedings of the
  American Mathematical Society 10 (1959) 448--450.

\bibitem{Com94}
P.~Combettes, Inconsistent signal feasibility problems: Least-squares solutions
  in a product space, USSR Computational Mathematics and Mathematical Physics
  42 (1994) 2955--2966.

\bibitem{ComBon99}
P.~Combettes, P.~Bondon, Hard-constrained inconsistent signal feasibility
  problems, IEEE Transactions on Signal Processing 47 (1999) 2460--2468.

\bibitem{YairAnd15}
Y.~Censor, A.~Cegielski, Projection methods: an annotated bibliography of books
  and reviews, Optimization 64~(11) (2015) 2343--2358.

\bibitem{censor2018algorithms}
Y.~Censor, M.~Zaknoon, Algorithms and convergence results of projection methods
  for inconsistent feasibility problems: a review, arXiv preprint
  arXiv:1802.07529.

\bibitem{drusvyatskiy2017note}
D.~Drusvyatskiy, G.~Li, H.~Wolkowicz, A note on alternating projections for
  ill-posed semidefinite feasibility problems, Mathematical Programming
  162~(1-2) (2017) 537--548.

\bibitem{behling2020infeasibility}
R.~Behling, Y.~Bello-Cruz, L.-R. Santos, Infeasibility and error bound imply
  finite convergence of alternating projections, arXiv preprint
  arXiv:2008.03354.

\bibitem{Mor06.1}
B.~S. Mordukhovich, Variational Analysis and Generalized Differentiation. {I}:
  {B}asic {T}heory, Vol. 330 of Grundlehren der Mathematischen Wissenschaften
  [Fundamental Principles of Mathematical Sciences], Springer, Berlin, 2006.

\bibitem{Eke79}
I.~Ekeland, Nonconvex minimization problems, Bull. Amer. Math. Soc. (N.S.)
  1~(3) (1979) 443--474.

\bibitem{Kru18}
A.~Y. Kruger, About intrinsic transversality of pairs of sets, Set-Valued Var.
  Anal. 26~(1) (2018) 111--142.

\bibitem{luke2020necessary}
D.~R. Luke, M.~Teboulle, N.~H. Thao, Necessary conditions for linear
  convergence of iterated expansive, set-valued mappings, Mathematical
  Programming 180~(1) (2020) 1--31.

\bibitem{rockafellar2009variational}
R.~T. Rockafellar, R.~J.-B. Wets, Variational Analysis, Vol. 317, Springer
  Science \& Business Media, 2009.

\end{thebibliography}
\end{document}